%% file: Preprint.tex
\documentclass[12pt
				,a4paper
				,twoside
				,final
				,USenglish
				]{amsart}

\input{Ressources/Preamble.tex}

\hypersetup{
	pdfauthor={Matthias Johann Steiner},
	pdftitle={Schemes of Finite Expansion and Universally Closed Curves},
	pdfsubject={Master's Thesis of Matthias Johann Steiner},
	pdfkeywords={algebra, algebraic geometry, curves, algebraic curves, curves over fields, schemes of finite expansion, universally closed curves},
	pdfproducer={Latex with hyperref},
	pdfcreator={pdflatex}
}

\title{Schemes of Finite Expansion and Universally Closed Curves}
\author[M. Steiner]{Matthias Johann Steiner}
\address{Matthias Steiner - 9020 Klagenfurt am W\"orthersee, \"Osterreich}
\email{steiner.matthias@gmx.at}

\begin{document}
	
	\begin{abstract}
		\noindent In algebraic geometry there is a well-known categorical equivalence between the category of normal proper integral curves over a field $k$ and the category of finitely generated field extensions of $k$ of transcendence degree $1$. In this paper we generalize this equivalence to the category of normal quasi-compact universally closed separated integral $k$-schemes of dimension $1$ and the category of field extensions of $k$ of transcendence degree $1$. Our key technique are morphisms of finite expansion which can be considered as relaxation of morphisms of finite type. Since the schemes in the generalized category have many properties similar to normal proper integral curves, we call them normal integral universally closed curves over $k$.
		
	\end{abstract}
	
	\maketitle
	
	\section*{Introduction}
	The following theorem is a well-known result in the theory of curves.
	\begin{Th*}[{\autocite[Thm.~15.21]{AlgGeom-Goertz}}]\label{Th: classification of curves over fields}
		Let $k$ be a field. There is a contravariant equivalence between the categories of
		\begin{enumerate}[label=(\roman*)]
			\item normal proper integral curves over $k$ (with non-constant morphisms),
			
			\item extension fields $K$ of $k$, finitely generated and of transcendence degree $1$ (with $k$-homomorphisms),
		\end{enumerate}
		given by mapping a curve $C$ as in (i) to its function field $K(C)$.
	\end{Th*}
	In this paper we will extend this equivalence of categories to arbitrary field extensions of transcendence degree $1$. To do so we will have to relax the property \textit{of finite type} of schemes in (i). Although this property is utilized in several critical steps of the proof of the theorem it limits us to finitely generated field extensions of $k$. In \autocite{Hamacher-FinExp} Paul Hamacher introduced a suitable relaxation for \textit{of finite type}, he called it \textit{of finite expansion}. In a nutshell a $R$-algebra $A$ is \textit{of finite expansion}, if the structure homomorphism decomposes into a homomorphism of finite presentation followed by an integral homomorphism. Indeed, every finite type $k$-algebra is of finite expansion by Noether normalization, but we will see that the new notion \textit{of finite expansion} also contains many non-finite type algebras.
	
	With morphisms \textit{of finite expansion} we will be able to prove the following theorem.
	\begin{Th*}[see \Cref{Th: Main Theorem}]\label{Main Result}
		Let $k$ be a field. There is a contravariant equivalence between the categories of 
		\begin{enumerate}[label=(\roman*)]
			\item normal quasi-compact universally closed separated integral $k$-schemes of dimension $1$ (with non-constant morphisms),
			
			\item extension fields $K$ of $k$ of transcendence degree $1$ (with \newline$k$-homomorphisms),
		\end{enumerate}
		given by mapping a scheme $X$ as in (i) to its function field $K(X)$.
	\end{Th*}
	During the development of the proof we will see that the schemes in (i) have many properties very similar to normal proper integral curves. Therefore, we will call the schemes in (i) normal integral universally closed curves.
	
	This paper is a condensed version of the author's master's thesis \autocite{Steiner-Thesis}.
	
	\subsubsection*{Acknowledgments}
	I would like to thank my master's thesis advisor, Dr.\ Paul Hamacher, for his support, explanations and patience.

	\section{Morphisms of finite expansion}
	In his preprint \autocite{Hamacher-FinExp} Paul Hamacher developed an \`etale cohomology theory for universally closed morphism of schemes. Previously it was only possible to define the \`etale cohomology group with compact support of a scheme if the scheme is of finite
	type. With his new cohomology theory Hamacher also introduced a new notion of morphisms of schemes, namely morphisms of finite expansion. Morphisms of finite expansion generalize morphisms of finite type, they can be viewed as relaxation of the finiteness condition. Nevertheless, they behave very similar to morphisms of finite type. One of the most important of these similarities is that they are compactifiable. Compactification in our sense means that we can decompose a separated morphism of finite expansion into an open immersion followed by a universally closed separated morphism.

	Since \autocite{Hamacher-FinExp} is only available as preprint, we will restate all necessary definitions and results for a comprehensive treatment.

	\subsection{Algebras of finite expansion}
	We start with the definition of finite expansion for algebras.
	\begin{Def}[{\autocite[Def.~1.1]{Hamacher-FinExp}}]\label{Def: Finite expansion}
		Let $R$ be a ring and $A$ be an $R$-algebra. A family $(a_i)_{i \in I}$ of elements in $A$ is called a quasi-generating system of $A$, if $A$ is integral over $R\left[a_i \mid i \in I\right]$; the $a_i$ are called quasi-generators. If there exists a finite quasi-generating system of $A$, we say that $A$ is of finite expansion over $R$.
	\end{Def}
	
	\begin{Rem}[{\autocite[Rem.~1.2]{Hamacher-FinExp}}]\label{Rem: finite pres + integral}
		Alternatively we could say an $R$-algebra $A$ is of finite expansion if and only if there exists an integral morphism $R[x_1,\dots,x_n] \to A$. In particular, the structure morphism decomposes into a morphism of finite presentation and an integral morphism.
	\end{Rem}
	
	\begin{Lem}[{\autocite[Lemma~1.3]{Hamacher-FinExp}}\protect\footnotemark]\label{Lem: simple properties for algebras of finite expansion}
		We fix a ring $R$ and an $R$-algebra $A$.
		\begin{enumerate}[label=(\arabic*)]
			\item Let $B$ be an $A$-algebra and assume that $A$ is of finite expansion over $R$. Then $B$ is of finite expansion over $R$ if and only if it is of finite expansion over $A$.
			\item Let $f_1,\dots,f_m \in A$ such that $(f_1,\dots,f_m)_A = A$. Then $A$ is of finite expansion over $R$ if and only if $A_{f_i}$ is of finite expansion over $R$ for every $i$.
			\item Let $R'$ be another $R$-algebra and assume that $A$ is of finite expansion over $R$. Then $R' \otimes_R A$ is finite expansion over $R'$.
			\item Let $R'$ be a faithfully flat $R$-algebra and assume that $R' \otimes A$ is of finite expansion over $R'$. Then $A$ is of finite expansion over $R$.
		\end{enumerate}
	\end{Lem}\footnotetext{A spelling mistake in the original statement was corrected.}
	We provide two examples that should convince the reader that our new notion \textit{of finite expansion} includes non-finite type algebras.
	\begin{Ex}\label{Ex: examples of algebras of finite expansion}
		Let $k$ be a field.
		\begin{enumerate}[label=(\arabic*)]
			\item\label{Ex 1} The $k$-domain $k[\, x^\frac{1}{n} \mid n \in \mathbb{N} \,]$ is of finite expansion over $k$ and $x$ is the quasi-generator.
			\item\label{Ex 2} Let $x$ be a transcendental element over $k$, and let $k(x) \subset K$ be an algebraic field extension. Then the integral closure of $k[x]$ in $K$ is of finite expansion over $k$.
		\end{enumerate}
	\end{Ex}
	
	\subsection{Schemes of finite expansion}
	We now generalize \Cref{Def: Finite expansion} to schemes and introduce the first fundamental results for morphisms of finite expansion.
	\begin{CorDef}[{\autocite[Cor./Def.~1.4]{Hamacher-FinExp}}]\label{Def: Schemes of finite expansion}
		We call a morphism $f: X \to Y$ of schemes locally of finite expansion if the following equivalent conditions are satisfied.
		\begin{enumerate}
			\item[(a)] For every affine open subscheme $V \subset Y$ and every affine open subscheme $U \subset f^{-1}(V)$, the $\mathcal{O}_Y(V)$-algebra $\mathcal{O}_X(U)$ is of finite expansion.
			\item[(b)] There exists a covering $Y = \bigcup V_i$ by open affine subschemes $V_i \cong \Spec(R_i)$ and a covering $f^{-1}(V_i) = \bigcup U_{i,j}$ by open affine subschemes $U_{i,j} \cong \Spec A_{i,j}$ such that for all $i, j$ the $R_i$-algebra $A_{i,j}$ is of finite expansion.
		\end{enumerate}
		We say that a morphism $f: X \to Y$ is of finite expansion if it is locally of finite expansion and quasi-compact.
	\end{CorDef}
	\begin{Cor}[{\autocite[Cor.~1.5]{Hamacher-FinExp}}]\label{Cor: properties of morphisms of finite expansion}
		\begin{enumerate}[label=(\arabic*)]
			\item[]
			\item The properties \quotes{locally of finite expansion} and \quotes{of finite expansion} of morphisms of schemes are stable under composition, base change and faithfully flat descent, and are local on the target. The property \quotes{of finite expansion} is also local on the source.
			\item Let $f: X \to Y$ and $g: Y \to Z$ be morphisms of schemes. If $g \circ f$ is locally of finite expansion (resp.\ of finite expansion and $g$ is quasi-separated), then $f$ is locally of finite expansion (resp.\ of finite expansion).
		\end{enumerate}
	\end{Cor}
	If we impose mild conditions, it is also possible to decompose a morphism of finite expansion into an integral morphism and a morphism of finite presentation. 
	\begin{Prop}[{\autocite[Prop.~1.8]{Hamacher-FinExp}}]\label{Prop: properties of morphism of finite expansion}
		Let $f: X \to Y$ be a separated morphism of qcqs schemes.
		\begin{enumerate}[label=(\arabic*)]
			\item If $f$ is universally closed then it can be decomposed as $f = h \circ g$ with $g$ integral and $h$ proper.
			\item If $f$ is of finite expansion then it can be decomposed as $f = h \circ g$ with $g$ integral and $h$ of finite presentation.
		\end{enumerate}
	\end{Prop}
	A key result for our paper is the following corollary.
	\begin{Cor}[{\autocite[Cor.~1.9]{Hamacher-FinExp}}] \label{Cor: sep univ closed morph => finite exp}
		Every separated universally closed morphism between qcqs schemes is of finite expansion.
	\end{Cor}
	As already mentioned in the introduction morphisms of finite expansion are compactifiable in the following sense.
	\begin{Th}[{\autocite[Thm.~1.17]{Hamacher-FinExp}}]\label{Th: compactification of morphism of finite expansion}
		Let $f: X \to Y$ be a separated morphism of finite expansion between qcqs schemes. Then $f$ can be written as composition $f = \bar{f} \circ j$ where $j: X \to \bar{X}$ is an open embedding and $\bar{f}: \bar{X} \to Y$ is separated and universally closed.
	\end{Th}
	We conclude this section with the definition of universally closed curves over a field. We introduce this definition to simplify notation and to emphasize the relation between normal proper integral curves and normal quasi-compact universally closed separated integral schemes of dimension $1$ in the coming sections.
	\begin{Def}[Universally closed curves]\label{Def: universally close curve}
		Let $k$ be a field. A non-empty connected $k$-scheme $X$ is called a universally closed curve if it satisfies the following conditions.
		\begin{enumerate}[label=(\roman*)]
			\item $X$ is quasi-compact, universally closed, separated and of dimension $1$.
			\item $\trdeg_k \big(\kappa(\eta)\big) = 1$ for every generic point $\eta$ of an irreducible component of $X$.
		\end{enumerate}
	\end{Def}
	By \Cref{Cor: sep univ closed morph => finite exp} every universally closed curve is of finite expansion over $\Spec(k)$ and we will often use this fact without explicitly referring to the corollary. 
	
	\subsubsection{Extending morphisms of finite expansion}
	For schemes locally of finite presentation it is well-known that one can extend a local ring homomorphism to a morphism of schemes (cf.\ \autocite[Prop.~10.52]{AlgGeom-Goertz} and \autocite[\href{https://stacks.math.columbia.edu/tag/0BX6}{Tag 0BX6}]{stacks-project}). We will extend this property to schemes locally of finite expansion.
	\begin{Lem}\label{Lem: split of ring homomorphism}
		Let $R$ be a ring, and let $A$ and $B$ be $R$-domains. Assume that $A$ is of finite expansion, that $B$ is a normal, and let $\mathfrak{p} \subset B$ be a prime ideal. Suppose a homomorphism $\phi: A \to B_\mathfrak{p}$ is given. Then $\phi$ splits as $A \to B_f \to B_\mathfrak{p}$ for some $f \in B \setminus \mathfrak{p}$. 
	\end{Lem}
	\begin{proof}
		For simplicity we assume that $\phi$ is injective. By assumption we can find elements $a_1,\dots,a_n \in A$ such that $R[a_1,\dots,a_n] \to A$ is an integral ring homomorphism. Via $\phi$ we have that $a_i \in B_\mathfrak{p}$ for all $i$, thus we can write $a_i$ as 	$a_i = \frac{x_i}{y_i}$, where $x_i \in B$ and $y_i \in B \setminus \mathfrak{p}$. We define the element $f \in B$ as $f = \prod_{i = 1}^{n}y_i$. It is easy to see that $\phi |_{R[a_1,\dots,a_n]}$ factors through $B_f$, and by \autocite[Prop.~8.10]{Kemper-CommAlg} $B_f$ is also normal. Denote with $C$ the integral closure of $R[a_1,\dots,a_n]$ in $\Frac(B)$. It is obvious that $A \subset C$ and by normality of $B_f$ we must also have that $C \subset B_f$. Thus we have found the following split for $\phi$:
		\begin{equation*}
			\begin{tikzcd}
				A \arrow[hookrightarrow]{r}\arrow[bend left, hookrightarrow]{urrd}{\phi} & B_f \arrow[hookrightarrow]{r} & B_\mathfrak{p}.
			\end{tikzcd}
		\end{equation*}
		If $\phi$ is not injective, we replace $A$ by $\phi(A)$ in the above arguments. This concludes the proof.
	\end{proof}
	In the next proposition we provide the geometric formulation of this lemma. We omit the proof.
	\begin{Prop}\label{Prop: extension of morphism of finite expansion}
		Let $X$ and $Y$ be integral $S$-schemes, and let $x \in X$, $y \in Y$ be points lying over the same point $s \in S$. Suppose that $X$ is normal and that $Y$ is locally of finite expansion over $S$. Let $\phi_x: \mathcal{O}_{Y,y} \to \mathcal{O}_{X,x}$ be a local $\mathcal{O}_{S,s}$-homomorphism. Then there exists an open neighborhood $U$ of $x$ and a $S$-morphism $f: U \to Y$ with $f(x) = y$ and such that the homomorphism $\mathcal{O}_{Y,y} \to \mathcal{O}_{X,x}$ induced by $f$ is $\phi_x$.
	\end{Prop}
	\begin{proof}
		Replacing $S$, $X$ and $Y$ by suitable affine open subschemes, we may assume that $S = \Spec(R)$, $X = \Spec(B)$ and $Y = \Spec(A)$ are affine. The points $x$ and $y$ correspond to prime ideals $\mathfrak{q} \subset B$ and $\mathfrak{p} \subset A$. Via $\phi_x$ we immediately obtain the following $R$-homomorphism 
		\begin{equation*}
			\begin{tikzcd}
				A \arrow[hookrightarrow]{r} & A_\mathfrak{p} \arrow[r, "\phi_x"] & B_\mathfrak{q}.
			\end{tikzcd}
		\end{equation*}
		By \Cref{Lem: split of ring homomorphism} this homomorphism factors through $\phi: A \to B_t$ for some $t \in B \setminus \mathfrak{q}$. It is clear that the corresponding morphism $D_{\Spec(B)}(t) \to \Spec(A)$ maps $x$ to $y$ and induces the morphism $\phi_x$ on the stalks. 
	\end{proof}

	\subsection{Quasi-generating systems and polynomial rings}
	Let $k$ be a field, and let $A$ be a non-empty $k$-algebra of finite expansion. By assumption we can find elements $a_1,\dots,a_n \in A$ such that $k[a_1,\dots,a_n] \subset A$ is an integral extension of $k$-algebras. Via Noether normalization \autocite[Thm.~8.19]{Kemper-CommAlg} we can find algebraically independent elements $a_1',\dots,a_m' \in k[a_1,\dots,a_n]$ such that $k[a_1',\dots,a_m'] \subset k[a_1,\dots,a_n]$ is also an integral extension. Thus, by the tower property of integral extensions \autocite[Cor.~8.6]{Kemper-CommAlg} $(a_1',\dots,a_m')$ is an algebraically independent quasi-generating system of $A$. Further $k[a_1',\dots,a_m'] \cong k[x_1,\dots,x_m]$, where the latter ring denotes the polynomial ring in $m$ variables over $k$. Let $(b_1,\dots,b_l)$ be another algebraically independent quasi-generating system of $A$. Then by \autocite[Cor.~8.13]{Kemper-CommAlg} we have the following equality:
	\begin{equation*}
		m = \dim(k[a_1',\dots,a_m']) = \dim(A)= \dim(k[b_1,\dots,b_l]) = l.
	\end{equation*}
	Thus all algebraically independent quasi-generating systems are of the same size.
	
	This observation justifies the following assumption for our notation in the coming chapters: If $\mathfrak{a}$ is a quasi-generating system of $A$, then we can always assume that $\mathfrak{a}$ is algebraically independent over $k$. Further, $k[\mathfrak{a}]$ is isomorphic to a polynomial ring over $k$. To exemplify this isomorphism we will denote the elements of $\mathfrak{a}$ by indeterminate variables of a polynomial ring. I.e., we will write $\mathfrak{a}=(x_1,\dots,x_n)$.
	
	Equipped with our new notation we conclude this section by demonstrating that a $k$-algebra of finite expansion is a Jaffard ring.
	\begin{Lem}\label{Lem: domain of finite expansion is Jaffard domain}
		Let $k$ be a field, and let $A$ be a $k$-algebra of finite expansion. Then $A$ is a Jaffard ring, i.e., $\dim(A[x_1,\dots,x_n]) = n + \dim(A)$.
	\end{Lem}
	\begin{proof}
		First note that if $R \subset S$ is an integral extension of rings, then $R[x] \subset S[x]$ is also an integral extension of rings.
		
		Let $m = \dim(A)$. By assumption we can find algebraically independent elements $x_1,\dots,x_m \in A$ such that $k[x_1,\dots,x_m] \subset A$ is an integral extension. Now let $n \geq 1$ and consider the polynomial ring $A[y_1,\dots,y_n]$. By our remark at the beginning
		\begin{equation*}
			k[x_1,\dots,x_m]\big[y_1,\dots,y_n\big] \subset A[y_1,\dots,y_n]
		\end{equation*}
		is also an integral extension. Further we have that 
		\begin{equation*}
			k[x_1,\dots,x_m]\big[y_1,\dots,y_n\big] \cong k[y_1\dots,y_n,y_{n+1},\dots,y_{m+n}],
		\end{equation*}
		and combined with \autocite[Cor.~5.7,Cor.~8.13]{Kemper-CommAlg} we conclude that
		\begin{equation*}
			\dim(A[y_1,\dots,y_m]) = \dim(k[y_1\dots,y_n,y_{n+1},\dots,y_{m+n}]) = m + n. \qedhere
		\end{equation*}
	\end{proof}

	\section{Schemes of finite expansion over fields}
	In the first part of this section we introduce dimension formulae for schemes of finite expansion over a field $k$. In the second part we demonstrate that a normal $k$-domain of finite expansion is a Pr\"ufer domain. We will then use this property to show that for a normal integral universally closed curve over $k$ we have a bijection between the closed points of the curve and valuation rings inside the function field.
	 
	\subsection{Dimension of schemes of finite expansion}
	For an affine scheme of finite expansion it is straight forward to compute its dimension.
	\begin{Prop}\label{Prop: Dim. of aff. scheme over a field}
		Let $k$ be a field, and let $X = \Spec(A)$ be an affine $k$-scheme of finite expansion. Then we have that $\dim(X) = d$, where $d$ is the number of algebraically independent quasi-generators of $A$.
	\end{Prop}
	\begin{proof}
		By assumption we can find algebraically independent elements $x_1,\dots,x_d \allowbreak\in A$ such that $k[x_1,\dots,x_d] \subset A$ is an integral extension of rings. The claim follows then from \autocite[Cor.~5.7,Cor.~8.13]{Kemper-CommAlg}.
	\end{proof}
	Before investigating the general case, we observe that the property (locally) of finite expansion can be passed to the reduced subscheme.
	\begin{Lem}\label{Lem: X_red also of finite expansion}
		Let $k$ be a field.
	\begin{enumerate}[label=(\arabic*)]
		\item Let $A$ be a $k$-algebra of finite expansion. Then $A / \mathcal{N}$ is also of finite expansion over $k$, where $\mathcal{N}$ denotes the nilradical of $A$.
		\item Let $X$ be a $k$-scheme (locally) of finite expansion. Then $X_{red}$ is also (locally) of finite expansion over $k$.
	\end{enumerate}
	\end{Lem}
	\begin{proof}
		For (1), by assumption we have an integral ring extension $f:k[x_1,\dots,x_n] \allowbreak\hookrightarrow A$ for some $n \in \mathbb{N}$. Via the projection $A \twoheadrightarrow A / \mathcal{N}$ we can extend $f$ to an integral homomorphism $\bar{f}: k[\bar{x}_1,\dots,\bar{x}_n] \hookrightarrow A / \mathcal{N}$, where $\bar{x}_i$ denotes the projection of $x_i$ into $A / \mathcal{N}$.
		
		For (2), the reduced subscheme is defined as $X_{red} =(X, \mathcal{O}_X / \mathcal{N})$, where $\mathcal{N}$ is the nilradical of $\mathcal{O}_X$. The claim now follows from \Cref{Def: Schemes of finite expansion} and part (1).
	\end{proof}
	The following theorem connects the dimension of an integral $k$-scheme locally of finite expansion with the transcendence degree of its function field. It is an adaptation of the locally of finite type case presented in \autocite[Thm.~5.22]{AlgGeom-Goertz}.
	\begin{Th}\label{Thm: trdeg = dim}
		Let $k$ be a field. Let $X$ be an irreducible $k$-scheme locally of finite expansion with generic point $\eta$.
		\begin{enumerate}[label=(\arabic*)]
			\item $\dim(X) = \trdeg_k \big(\kappa(\eta)\big)$.
			\item Let $f: Y \to X$ be a morphism of $k$-schemes of finite expansion such that $f(Y)$ contains the generic point $\eta$ of $X$. Then $\dim(Y) \geq \dim(X)$. In particular we have $\dim(U) = \dim(X)$ for any non-empty open subscheme $U$ of $X$.
		\end{enumerate}
	\end{Th}
	\begin{proof}
		(1) We may assume that $X$ is reduced, and covering $X$ by non-empty open affine subschemes $U$ we may assume that $X = \Spec(A)$, where $A$ is a $k$-domain of finite expansion. Thus we have that $\kappa(\eta) = \Frac(A)$. Let $(x_1,\dots,x_d)$ be an algebraically independent quasi-generating system of $A$, then $k[x_1,\dots,x_d] \subset A$ is an integral extension of integral domains, hence $\kappa(\eta)$ is an algebraic field extension of $K = k(x_1,\dots,x_d)$. With the tower property of the transcendence degree (cf.\ \autocite[Satz~23.4]{Algebra-Karpf}) we conclude that:
		\begin{equation*}
			\trdeg_k\big(\kappa(\eta)\big) = \trdeg_k(K) + \trdeg_K\big(\kappa(\eta)\big) = \trdeg_k(K) = d.
		\end{equation*}
		By \Cref{Prop: Dim. of aff. scheme over a field} we also have that $\dim(A) = d$, so the claim $\dim(A) = \trdeg_k \big(\kappa(\eta)\big)$ follows.
		
		(2) By hypothesis there exists $\theta \in Y$ such that $f(\theta) = \eta$. Therefore $f$ induces a $k$-embedding $\kappa(\eta) \hookrightarrow \kappa(\theta)$. Denote with $Z$ the closure of $\theta$. Then
		\begin{equation*}
			\dim(X) = \trdeg_k\big( \kappa(\eta) \big) \leq \trdeg_k\big( \kappa(\theta) \big) = \dim(Z) \leq \dim(Y). \qedhere
		\end{equation*}
	\end{proof}
	With this theorem we can immediately conclude that for an integral universally closed curve $X$ over $k$ we have that $\trdeg_k\big( K(X) \big) = 1$.
	
	Similar one can extend dimension formulae for products and extensions of the base field. Since the proof is analog to \autocite[Prop.~5.37, 5.38]{AlgGeom-Goertz} we skip it.
	\begin{Prop}
		Let $k$ be field. Let $X$, $Y$ be non-empty $k$-schemes locally of finite expansion, and let $K$ be a field extension of $k$. Then 
		\begin{enumerate}[label=(\arabic*)]
			\item $\dim(X \times_k Y) = \dim(X) + \dim(Y)$,
			
			\item $\dim(X) = \dim(X \otimes_k K)$.
		\end{enumerate}
	\end{Prop}
	
	\subsection{Normal one-dimensional domains of finite expansion are Pr\"ufer domains}
	Let $C$ be a normal proper integral curve over a field $k$. It is well-known that for a closed point $x \in C$ the local ring $\mathcal{O}_{C,x}$ is a discrete valuation ring (cf.\ \autocite[Rem.~15.23]{AlgGeom-Goertz}). In this section we will establish a similar result for schemes of finite expansion over $k$.
	\begin{Th}\label{Th: localization at max. ideal is valuation ring}
		Let $k$ be a field, and let $A$ be a normal $k$-domain of finite expansion of dimension $1$. Then $A$ is a Pr\"ufer domain. I.e., if $\mathfrak{m} \subset A$ is a maximal ideal, then $A_\mathfrak{m}$ is a valuation ring.
	\end{Th}
	\begin{proof}
		This is an application of \autocite[Ch.~III,\S1,Thm. 1.2]{Fuchs-2001-Modules}.
	\end{proof}
	We provide counterexamples that normality is a necessary assumption and that in general we do not obtain a discrete valuation ring.
	\begin{Ex}\label{Ex: counterexample for normality}
		Let $k$ be a field.
		\begin{enumerate}
			\item The $k$-algebra $A = k[x^2,x^3] \subset k[x]$ is of finite expansion and of dimension $1$. It is not normal, because $x \in \Frac(A)$ is integral over $A$ however $x \not \in A$. If we localize $A$ at $\mathfrak{m} = (x^2,x^3)$ we again have that $x, x^{-1} \not \in A_\mathfrak{m}$, so $A_\mathfrak{m}$ is not a valuation ring.
			
			\item Let $B = k[\,x^\frac{1}{n} \mid n \in \mathbb{N}\,]$. Then $B$ is of finite expansion, normal and of dimension $1$. Consider the maximal ideal $\mathfrak{m} = (\,x^\frac{1}{n} \mid n \in \mathbb{N}\,)$. By \Cref{Th: localization at max. ideal is valuation ring} $B_\mathfrak{m}$ is a valuation ring and it has value group $\mathbb{Q}$. Thus $B_\mathfrak{m}$ is not a discrete valuation ring.
		\end{enumerate}
	\end{Ex}
	In the next Corollary we provide the geometric formulation of \Cref{Th: localization at max. ideal is valuation ring}.
	\begin{Cor}\label{Cor: integral normal scheme of finite expansion dim 1 => local rings are valuation rings}
		Let $k$ be a field, and let $X$ be a $k$-scheme which is normal, integral, of finite expansion and of dimension $1$. Let $x \in X$ be a closed point. Then $\mathcal{O}_{X,x}$ is a valuation ring.
	\end{Cor}
	\begin{proof}
		This is an application of \Cref{Thm: trdeg = dim} (2) and \Cref{Th: localization at max. ideal is valuation ring}.
	\end{proof}
	
	\subsection{Closed points correspond to valuation rings}
	Let $C$ be a normal proper integral curve over a field $k$, then one has a bijection between the closed points of $C$ and discrete valuation rings inside the function field $K(C)$ (cf.\ \autocite[Rem.~15.23]{AlgGeom-Goertz}). We can extend this result in a similar fashion to universally closed curves.
	\begin{Th}\label{Th: closed points correspond to valuation rings}
		Let $k$ be a field, and let $X$ be a normal integral universally closed curve over $k$. Then one has a bijection between the sets
		\begin{equation*}
			\left\{\text{closed points of }X\right\}\xlongleftrightarrow{1:1}\left\{\begin{matrix}\text{valuation rings}\\\mathcal{O} \subset K(X)\\ \text{with }k^\times \subset \mathcal{O}^\times\end{matrix}\right\}.
		\end{equation*}
	\end{Th}
	\begin{proof}
	First we prove that every valuation ring in $K(X)$ comes from a closed point. \\
	Let $\mathcal{O} \subset K(X)$ be a valuation ring with $k^\times \subset \mathcal{O}^\times$. Then we can construct a commutative diagram
	\begin{equation*}
		\begin{tikzcd}
			\Spec(K(X)) \arrow[r, "\iota"]\arrow[d] & X\arrow[d] \\
			\Spec(\mathcal{O}) \arrow[r] & \Spec(k),
		\end{tikzcd}
	\end{equation*}
	where $\iota$ maps the unique point of $\Spec(K(X))$ to the generic point of $X$. By the generalized valuative criterion (cf.\ \autocite[Thm.~15.8]{AlgGeom-Goertz}) there exists a unique morphism $v: \Spec(\mathcal{O}) \to X$, which preserves commutativity in the diagram. Suppose the unique closed point $z \in \Spec(\mathcal{O})$ maps to the closed point $x \in X$, i.e., $v(z) = x$. Then we obtain the following local homomorphism of local rings
	\begin{equation*}
		\begin{split}
			&v^\sharp_z: (v^{-1}\mathcal{O}_X)_z \to \mathcal{O}_{\Spec(\mathcal{O}),z} \\
			\Rightarrow \; &v^\sharp_z: \mathcal{O}_{X,x} \to \mathcal{O}.
		\end{split}
	\end{equation*}
	For commutativity of the above diagram we must have that the generic point of $\mathcal{O}$ is mapped to the generic point of $X$. I.e., $v$ is a dominant morphism, but this makes $v^\sharp_z$ into an injective local homomorphism (see \autocite[\href{https://stacks.math.columbia.edu/tag/0CC1}{Tag 0CC1}]{stacks-project}). Hence $\mathcal{O}$ dominates $\mathcal{O}_{X,x}$. From \Cref{Cor: sep univ closed morph => finite exp} and \Cref{Cor: integral normal scheme of finite expansion dim 1 => local rings are valuation rings} it follows that $\mathcal{O}_{X,x}$ is also a valuation ring. Valuation rings are maximal among the domination order, thus $\mathcal{O} =\mathcal{O}_{X,x}$.
	
	Injectivity follows similar. Let $x,x' \in X$ be such that $\mathcal{O}_{X,x} \cong \mathcal{O}_{X,x'}$. Again by the valuative criterion we obtain a morphism $v: \Spec(\mathcal{O}_{X,x}) \cong \Spec(\mathcal{O}_{X,x'}) \to X$. If $z$ and $z'$ are the unique closed points of $\Spec(\mathcal{O}_{X,x})$ and $\Spec(\mathcal{O}_{X,x'})$, then we have that $x = v(z) = v(z') = x'$.
	\end{proof}
	\begin{Rem}\label{Rem: for a curve every valuation ring is discrete}
		From the theorem we can also conclude that for a normal proper integral curve $C$ over a field $k$ all valuation rings $\mathcal{O}$ contained in the function field $K(C)$ with $k^\times \subset \mathcal{O}^\times$ are discrete.
	\end{Rem}

	\section{Universally closed curves and extensions of transcendence degree 1}
	We have now developed all necessary tools to prove the main theorem. The proof is developed in a similar way as the original proof for curves presented in \autocite[Rem.~7.4.19]{EGAII}.
	
	\subsection{A homomorphism of function fields induces a morphism of universally closed curves}\label{Sec: A morphism between fields induces a morphism between schemes}
	On objects the association $X \mapsto K(X)$ is clear, but we must also check that there is a contravariant association of morphisms. In the first lemma we will establish this for dominant morphisms between integral schemes of finite expansion. Later in \Cref{Th: morphism constant or surjective and integral} we will prove that a morphism between integral universally closed curves is either constant or surjective. 
	\begin{Lem}\label{Lem: dominant morphism defines morphism of function fields}
		Let $k$ be a field, let $X$ and $Y$ be integral $k$-schemes of finite expansion, and let $f: X \to Y$ be a dominant morphism of $k$-schemes.
		\begin{enumerate}[label=(\arabic*)]
			\item $f$ induces a homomorphism of fields $f^*: K(Y) \to K(X)$ in a functorial way.
			\item  If $\dim(X) = \dim(Y)$, then $K(X)$ becomes via $f^*$ an algebraic field extension of $K(Y)$.
		\end{enumerate}
	\end{Lem}
	\begin{proof}
		For a dense morphism the generic point of $X$ is mapped to the generic point of $Y$, and the morphism $f^*$ is the induced morphism on stalks of the structure sheaves. It is a homomorphism of fields, hence injective. Under the assumptions of (2) the transcendence degrees over $k$ agree by \Cref{Thm: trdeg = dim} (1), thus the field extension is algebraic.
	\end{proof}
	The first major step is to establish that a homomorphism of function fields of universally closed curves is induced by a unique morphism of universally closed curves. This can be seen as adaption of \autocite[Cor.~7.4.13)]{EGAII} to universally closed curves.
	\begin{Th}\label{Th: morphism of fields comes from morphism of schemes}
		Let $k$ be a field. Let $X$ be a normal separated integral $k$-scheme of finite expansion of dimension $1$, and let $Y$ be an integral universally closed curve over $k$. Then every $k$-homomorphism $\alpha: K(Y) \to K(X)$ is of the form $f^*$ for a uniquely determined morphism $f: X \to Y$. 
	\end{Th}
	\begin{proof}
	We immediately obtain a morphism 
	\begin{equation*}
		\Spec(\alpha): \Spec(K(X)) \to \Spec(K(Y)) \to Y,	
	\end{equation*}
	 where the second morphism corresponds to the inclusion of the generic point of $Y$. Let $x \in X$ be a closed point and let $U = \Spec(B) \subset X$ be an affine open neighborhood of $x$. By assumption $U$ is normal, integral, of finite expansion and one-dimensional, further $x$ corresponds to a maximal ideal $\mathfrak{m}_x \subset B$. Hence by \Cref{Cor: integral normal scheme of finite expansion dim 1 => local rings are valuation rings} $\mathcal{O}_{X,x} \cong B_{\mathfrak{m}_x}$ is a valuation ring. Now we consider the following commutative diagram of schemes:
	\begin{equation}
		\begin{tikzcd}\label{Equ: valuative criterion diagram}
			\Spec(K(X))				\arrow[r, "\Spec(\alpha)"]\arrow[d]	&	Y	 \arrow[d]	\\
			\Spec(\mathcal{O}_{X,x})	\arrow[r]			&	\Spec(k).
		\end{tikzcd}
	\end{equation}
	By the generalized valuative criterion (see \autocite[Thm.~15.8]{AlgGeom-Goertz}) there exists a unique morphism $g_x: \Spec(\mathcal{O}_{X,x}) \to Y$.
	
	Let $V = \Spec(A)$ be an affine open neighborhood of finite expansion of $g_x(x)$ in $Y$. Then $g_x^{-1}(V)$ is open in $\Spec(\mathcal{O}_{X,x})$ and contains $x$, thus it is equal to $\Spec(\mathcal{O}_{X,x})$. So we obtain an induced homomorphism of rings $A \to \mathcal{O}_{X,x} \cong B_{\mathfrak{m}_x}$. By \Cref{Lem: split of ring homomorphism} the homomorphism $A \to B_{\mathfrak{m}_x}$ splits as $A \to B_\omega \to B_{\mathfrak{m}_x}$ for some $\omega \in B \setminus \mathfrak{m}_x$. Thus on the level of spectra we obtain an extension of $g_x$ to $\tilde{g}_x: D_U(\omega) \to V$, where $D_U(\omega)$ is an affine open neighborhood of $x$ in $X$. 
	
	For varying $x$ we would like to glue the extensions $\tilde{g}_x$ to a unique morphism $f: X \to Y$. According to \autocite[Prop.~3.5]{AlgGeom-Goertz} it is enough to show that $\tilde{g}_x: D_U(\omega) \to Y$ and $\tilde{g}_{x'}: D_{U'}(\omega') \to Y$ coincide on $D_U(\omega) \cap D_{U'}(\omega')$. Let us consider the equalizer $\Eq(\tilde{g}_{x},\tilde{g}_{x'})$, since $Y$ is separated the equalizer is a closed subscheme of $D_U(\omega) \cap D_{U'}(\omega')$ (cf.\ \autocite[Def.~and~Prop.~9.7]{AlgGeom-Goertz}). By construction $\tilde{g}_{x}$ and $\tilde{g}_{x'}$ preserve commutativity in Diagram \ref{Equ: valuative criterion diagram}, so we must have that $\Spec(K(X)) \subset \Eq(\tilde{g}_{x},\tilde{g}_{x'})$ and hence $\Eq(\tilde{g}_{x},\tilde{g}_{x'}) = D_U(\omega) \cap D_{U'}(\omega')$. Now we glue these extensions to a unique morphism $f: X \to Y$. 
	\end{proof}
	Analog to curves we can define the degree of a morphism between integral universally closed curves (cf.\ \autocite[p.~498]{AlgGeom-Goertz}).
	\begin{Def}[Degree of a morphism]\label{Def: degree of a morphism}
		Let $X$ and $Y$ be integral universally closed curves over a field $k$, and let $f: X \to Y$ be a morphism. If $f$ has dense image we define the degree of $f$ as $$\degree(f) := \Big[ K(X) : f^*\big(K(Y)\big) \Big].$$ If the image of $f$ is not dense we define the degree to be $0$.
	\end{Def} 
	The next corollary can be seen as generalization of \autocite[Cor.~7.4.16]{EGAII}.
	\begin{Cor}\label{Cor: morphism of degree 1 is an isomorphism}
		Let $k$ be a field. Let $f: X \to Y$ be a morphism of normal integral universally closed curves over $k$ which is of degree $1$. Then $f$ is an isomorphism.
	\end{Cor}
	\begin{proof}
		$f^\ast$ induces a field extension of degree $1$, hence it is an isomorphism. By \Cref{Th: morphism of fields comes from morphism of schemes}, its inverse comes from a unique morphism $g: Y \to X$. Using the equalizer we conclude that $g$ is indeed the inverse of $f$.
	\end{proof}
	
	\subsection{Properties of morphisms of universally closed curves}\label{Non-constant morphisms are of finite expansion}
	It is well-known that a non-constant morphism between normal proper integral curves is either constant or surjective, finite and flat (cf.\ \autocite[Prop.~15.16]{AlgGeom-Goertz} and \autocite[\href{https://stacks.math.columbia.edu/tag/0CCK}{Tag 0CCK}]{stacks-project}). We will now prove the analog statement for universally closed curves.
	\begin{Th}\label{Th: morphism constant or surjective and integral}
		Let $k$ be a field, and let $f: X \to Y$ be a morphism between integral universally closed curves over $k$. Then $f$ is constant or quasi-compact, separated, universally closed and surjective. With additional assumptions the following assertions hold if $f$ is not constant.
		\begin{enumerate}[label=(\arabic*)]
			\item If $X$ is normal then $f$ is integral.
			\item If $Y$ is normal then $f$ is flat.
		\end{enumerate}
	\end{Th}
	\begin{proof}
		Since $X$ and $Y$ are universally closed curves we can conclude by cancellation that $f$ is separated, quasi-compact and universally closed. In particular it follows that $f(X)$ is closed in $Y$. Images of irreducible sets under continuous maps are again irreducible, hence $f(X)$ is also irreducible. As $Y$ is integral and quasi-compact an irreducible closed subscheme is either a closed point or the whole scheme $Y$. Thus $f$ is either constant or surjective.
		
		Now let us prove the additional assertion (1). By \autocite[\href{https://stacks.math.columbia.edu/tag/01WM}{Tag 01WM}]{stacks-project} it is equivalent to show that $f$ is affine and universally closed. Recall that by \Cref{Lem: dominant morphism defines morphism of function fields} $f^\ast: K(Y) \hookrightarrow K(X)$ induces an algebraic field extension. Now let $\Spec(A) \subset Y$ be a non-empty affine open subscheme. Then $A$ is of finite expansion over $k$ and $\Frac(A) \cong K(Y) \subset K(X)$. Denote with $Z$ the normalization of $\Spec(A)$ in $K(X)$. Then by \autocite[Prop.~12.43]{AlgGeom-Goertz} $Z \cong \Spec(B)$, $K(Z) \cong K(X)$ and $B$ is also of finite expansion over $k$. So by \Cref{Th: morphism of fields comes from morphism of schemes} there exists a unique morphism $g: Z \to X$. By construction of $Z$ the image of $f \circ g$ lies in $\Spec(A)$, hence we have an induced morphism $g: Z \to f^{-1}(\Spec(A))$. To conclude the proof we must show that this is an isomorphism, so let us construct an inverse. 
		We choose an affine open covering $f^{-1}(\Spec(A)) = \bigcup_i \Spec(B_i)$. Then $\Frac(B_i) \cong K(X) \cong K(Z)$ for all $i$. Let $b \in B$, then $b$ is integral over $A$ and thus also over $B_i$. But $\Spec(B_i)$ is an affine open of $X$, so it is normal. Thus $b \in B_i $ and $B \subset B_i$ for all $i$. We obtain morphisms $h_i: \Spec(B_i) \to \Spec(B) = Z$ which are induced by inclusions of the coordinate rings into $K(X)$, thus we can glue them to a morphism $h: f^{-1}(\Spec(A)) \to Z$. Now we consider the equalizer $\Eq(h \circ g, \id_Z)$ which is a closed subscheme of $Z$. Since on function fields $h^\ast$ is inverse to $g^\ast$ we must have that $h \circ g (\eta_Z) = \id_Z(\eta_Z)$, but this implies that $\eta_Z \in \Eq(h \circ g, \id_Z)$ and thus $\Eq(h \circ g, \id_Z) = Z$. Therefore $h \circ g = id_Z$. As an open subscheme of a separated scheme $f^{-1}(\Spec(A))$ is also separated over $\Spec(k)$, therefore we can conclude by a symmetric argument that $g \circ h = \id_{f^{-1}(\Spec(A))}$.
		
		For assertion (2), pick points $x \in X$ and $y \in Y$ such that $f(x) = y$. The local ring $\mathcal{O}_{Y,y}$ is either a field or a valuation ring by \Cref{Cor: integral normal scheme of finite expansion dim 1 => local rings are valuation rings}. Further, $f$ is a dominant morphism between integral schemes, hence the induced homomorphism on local rings $\mathcal{O}_{Y,y} \to \mathcal{O}_{X,x}$ is injective by \autocite[\href{https://stacks.math.columbia.edu/tag/0CC1}{Tag 0CC1}]{stacks-project}. Therefore $\mathcal{O}_{X,x}$ is torsion free as a $\mathcal{O}_{Y,y}$-module and by \autocite[\href{https://stacks.math.columbia.edu/tag/0539}{Tag 0539}]{stacks-project} this proves that $\mathcal{O}_{X,x}$ is a flat $\mathcal{O}_{Y,y}$-module.
	\end{proof}
	
	\subsection{Construction of universally closed curves starting from field extensions}\label{Sec: Scheme construction starting from a field extension of transcendence degree 1}
	Starting with a field extension $K$ of $k$ of transcendence degree $1$ we will now construct a universally closed curve $X$ with $K(X) \cong K$. This construction can be seen as generalization of the first part of \autocite[Prop.~7.4.18]{EGAII}.
	\begin{Th}\label{Th: construction of scheme from function field}
		Let $k$ be a field, and let $K$ be a field extension of $k$ of transcendence degree $1$. Then there is a normal integral universally closed curve $X$ over $k$ with $K(X) \cong K$. It is unique up to isomorphism.
	\end{Th}
	\begin{proof}
		Let $x \in K$ be such that $x$ is transcendent over $k$. We denote with $\nu: U \to \Spec(k[x])$ the normalization of $\Spec(k[x])$ in $K$. By \autocite[Prop.~12.43]{AlgGeom-Goertz} the scheme $U$ has the following properties:
		\begin{enumerate}[label=(\arabic*)]
			\item The scheme $U$ is integral and normal, and $K(U) = K$.
			\item The morphism $\nu$ is integral and surjective and $\dim(U) = \dim(\Spec(k[x])) \allowbreak= 1$.
			\item $\Spec(k[x])$ is affine, thus $U = \Spec(A)$, where $A$ is the integral closure of $k[x]$ in $K$.
		\end{enumerate}
		Naturally we can regard $U$ also as $k$-scheme, we denote the $k$-structure morphism by $\pi: U \to \Spec(k).$
		From the above properties we can immediately conclude that $U$ is a one-dimensional, normal, quasi-compact, separated, integral $k$-scheme of finite expansion.
		
		Now we apply \Cref{Th: compactification of morphism of finite expansion} to write $\pi$ as composition $\pi = \widetilde{{\pi}} \circ j$, where $j: U \to X$ is an open immersion and $\widetilde{\pi}: X \to \Spec(k)$ is separated and universally closed. We consider $\widetilde{\pi}$ and $X$ as the compactification of $\pi$ and $U$ respectively. A priori it may not be clear that this is a meaningful notion of compactification for our purpose, therefore we will now establish that $j$ is dominant and that we can consider $X$ to be integral, normal and quasi-compact. 
		\begin{description}
			\item[\boldmath$j$ is dominant:] We restate the arguments presented in \autocite[Rem,~4.2]{conrad2007deligne}: $\pi$ is quasi-compact and $\widetilde{\pi}$ is separated, so by cancellation $j$ is also quasi-compact. Thus, the scheme-theoretic closure of $U$ in $X$ exists and we rename it as $X$. Obviously a closed subscheme of a separated and universally closed scheme is also separated and universally closed. So $j$ is dominant.
			
			\item[\boldmath$X$ is integral:] Closures and images under continuous maps of irreducible sets are again irreducible. Thus $X = \overline{j(U)}$ is irreducible.
			
			As $U$ is reduced $j$ factors as (\autocite[Prop.~3.51]{AlgGeom-Goertz})
			\begin{equation*}
				\begin{tikzcd}
					U \arrow[r, "j_{red}"]\arrow[bend right, drru, swap, "j"] & X_{red} \arrow[r, "\iota"] & X,
				\end{tikzcd}
			\end{equation*}
			where $\iota: X_{red} \to X$ is the canonical inclusion of the reduced subscheme. As the topological spaces of $X$ and $X_{red}$ agree $j_{red}$ must also be dense. We have an isomorphism between $U$ and $j(U)$, so $j(U)$ is reduced too. Therefore $\iota$ defines an isomorphism between $j(U)$ and $j_{red}(U)$. Hence $j_{red}$ is also an open immersion. The reduced structure morphism $\widetilde{\pi}_{red}: X_{red} \to \Spec(k)$ is again separated and universally closed by cancellation. Now we rename $X_{red}$ as $X$.
			
			\item[\boldmath$X$ is normal:] Denote with ${\pi': X_{norm} \to X}$ the normalization of $X$ in $K(X)$. By \autocite[Prop.~12.44]{AlgGeom-Goertz} the morphism $\pi'$ is integral and dominant and we obtain a unique morphism $j': U \to X_{norm}$ such that $j = \pi' \circ j'$. Also note that by \autocite[Rem.~12.46]{AlgGeom-Goertz} the restriction ${\pi'^{-1}(j(U)) \to j(U)}$ is an isomorphism. The open immersion $j$ defines an isomorphism between $U$ and $j(U)$ and $j'$ defines an isomorphism between $U$ and $j'(U)$. The composition of isomorphisms is an isomorphism, so we have that 
			\begin{equation*}
				U \cong j(U) \cong j'(U) \cong \pi'^{-1}\big(j(U)\big).
			\end{equation*}
			I.e., $j'(U)$ is isomorphic to an open subscheme of $X_{norm}$, therefore $j'$ is an open immersion. Further, it is also clear that $j'$ is dominant, else the equality 
			\begin{equation*}
				\overline{j(U)} = X = \overline{\pi'\big(j'(U)\big)}
			\end{equation*} 
			would not hold. Integral morphisms are separated and universally closed and both properties are stable under composition, thus $\widetilde{{\pi}}_{norm}: X_{norm} \to \Spec(k)$ is also separated and universally closed. Now we can rename $X_{norm}$ as $X$ and $j'$ as $j$.
			
			\item[\boldmath$X$ is quasi-compact:] A universally closed morphism is also quasi-compact by \autocite[\href{https://stacks.math.columbia.edu/tag/04XU}{Tag 04XU}]{stacks-project}. So $\widetilde{\pi}$ is quasi-compact.
		\end{description}
		
		Finally, we have the following equalities for dimensions and function fields 
		\begin{equation*}
			\begin{split}
				\dim(X_{norm}) = \dim(X) = \dim(U) = 1, \\
				K(X_{norm}) \cong K(X) \cong K(U) \cong K.
			\end{split}
		\end{equation*}		
		To sum it up, we have constructed a $k$-scheme $X$ which is quasi-compact, universally closed, separated, integral, normal, of
		dimension $1$ and with $K(X) \cong K$.
		
		Let $X'$ be another such scheme, then we have an isomorphism $K(X) \cong K \cong K(X')$. By \Cref{Th: morphism of fields comes from morphism of schemes} and \Cref{Cor: morphism of degree 1 is an isomorphism} it induces an isomorphism $X' \to X$.
	\end{proof}

	\subsection{Proof of the main theorem}
	We are now able to prove the central result of this thesis.
	\begin{Th}\label{Th: Main Theorem}
		Let $k$ be a field. There is a contravariant equivalence between the categories of 
		\begin{enumerate}[label=(\roman*)]
			\item normal integral universally closed curves over $k$ (with non-constant morphisms),
			
			\item extension fields $K$ of $k$ of transcendence degree $1$ (with \newline$k$-homomorphisms),
		\end{enumerate}
		given by mapping a scheme $X$ as in (i) to its function field $K(X)$.
	\end{Th}
	\begin{proof}
		The association $X \mapsto K(X)$ indeed defines a contravariant functor. For objects this is obvious, for morphisms we use \Cref{Lem: dominant morphism defines morphism of function fields} and \Cref{Th: morphism constant or surjective and integral}. Conversely, given a $k$-homomorphism $K(Y) \to K(X)$ of function fields we can construct with \Cref{Th: morphism of fields comes from morphism of schemes} a unique morphism $X \to Y$ of normal integral universally closed curves. By \Cref{Cor: morphism of degree 1 is an isomorphism} this is an isomorphism if the fields are isomorphic, so the functor is fully faithful. Essentially surjective follows from \Cref{Th: construction of scheme from function field}. So by \autocite[Prop.~7.26]{Category-Awodey} the functor gives rise to a contravariant categorical equivalence.
	\end{proof}

	\setcounter{biburllcpenalty}{9000}
	\setcounter{biburlucpenalty}{9000}
	\printbibliography	
\end{document}

%% file: Ressources/Preamble.tex
\pdfoutput=1

\usepackage[USenglish]{babel}
\usepackage{tikz-cd}
\usepackage{amsfonts}
\usepackage{amsmath}
\usepackage{amsthm}
\usepackage{amssymb}
\usepackage[latin2]{inputenc}
\usepackage[T1]{fontenc}
\usepackage{caption}
\numberwithin{equation}{section}
\usepackage[
			a4paper,
			hmarginratio={1:1},     
			vmarginratio={1:1},     
			textwidth=400pt,        
			heightrounded,  
			]{geometry}
\usepackage{enumitem}
\usepackage{mathtools}
\usepackage{hyperref}
\usepackage{thmtools}
\usepackage{extarrows}
\usepackage{lmodern}
\usepackage[final]{microtype}
\usepackage{csquotes}
\usepackage{cleveref}

\usepackage[backend=biber,
			style=alphabetic,
			]{biblatex}
\renewbibmacro*{doi+eprint+url}{%
	\iftoggle{bbx:doi}
	{\printfield{doi}}
	{}%
	\newunit\newblock
	\printfield{mrnumber}%
	\newunit\newblock
	\iftoggle{bbx:eprint}
	{\usebibmacro{eprint}}
	{}%
	\newunit\newblock
	\iftoggle{bbx:url}
	{\usebibmacro{url+urldate}}
	{}}

\DeclareSourcemap{
	\maps[datatype=bibtex, overwrite]{
		\map{
			\step[fieldset=address, null]
			\step[fieldset=location, null]
		}
	}
}

\theoremstyle{plain}
\newtheorem{Th}{Theorem}[section]
\newtheorem{Lem}[Th]{Lemma}
\newtheorem{Cor}[Th]{Corollary}
\newtheorem{Prop}[Th]{Proposition}
\newtheorem{CorDef}[Th]{Corollary/Definition}
\newtheorem{Def}[Th]{Definition}

\newtheorem*{Th*}{Theorem}
\newtheorem*{Prop*}{Proposition}
\newtheorem*{Def*}{Definition}
\newtheorem*{Not*}{Notation}

\makeatletter
\newtheoremstyle{mystyle}{}{}{\itshape}{}{\bfseries}{.}{0.5em}{\thmname{\@ifempty{#3}{#1}\@ifnotempty{#3}{#3}}}
\makeatother

\theoremstyle{definition}
\newtheorem{Rem}[Th]{Remark}
\newtheorem{Ex}[Th]{Example}

\theoremstyle{mystyle}

\makeatletter                          
\providecommand*{\toclevel@Th}{0}
\providecommand*{\toclevel@Prop}{0}
\providecommand*{\toclevel@Cor}{0}
\providecommand*{\toclevel@Def}{0}
\providecommand*{\toclevel@Named}{0}
\providecommand*{\toclevel@Lem}{0}
\providecommand*{\toclevel@Rem}{0}
\providecommand*{\toclevel@Ex}{0}
\providecommand*{\toclevel@CorDef}{0}
\makeatother

\newcommand{\quotes}[1]{``#1''}

\DeclareMathOperator{\Spec}{\mathop{\mathrm{Spec}}}
\DeclareMathOperator{\trdeg}{\mathop{\mathrm{trdeg}}}
\DeclareMathOperator{\Frac}{\mathop{\mathrm{Frac}}}
\DeclareMathOperator{\degree}{\mathop{\mathrm{deg}}}

\DeclareMathOperator{\id}{{\mathop{\mathrm{id}}}}
\DeclareMathOperator{\Eq}{\mathop{\mathrm{Eq}}}